\documentclass[11pt,leqno]{article}

\usepackage[active]{srcltx}

\usepackage{amsfonts,latexsym,amsmath,amssymb,amsthm}
\usepackage{fullpage}
\newtheorem{theorem}{Theorem}[section]
\newtheorem{lemma}[theorem]{Lemma}

\newtheorem{proposition}[theorem]{Proposition}

\newtheorem{remark}[theorem]{Remark}
\newtheorem{question}[theorem]{Question}

\theoremstyle{definition}

\def\R{{\mathbb R}}
\def\e{\varepsilon}

\theoremstyle{remark}

\newtheorem*{note*}{Note}
\numberwithin{equation}{section}
\makeatletter
\newcommand{\rank}{\mathop{\operator@font rank}}
\newcommand{\conv}{\mathop{\operator@font conv}}
\newcommand{\vol}{\mathop{\operator@font vol}}
\newcommand{\onetagright}{\tagsleft@false}
\makeatother
\newcommand{\ls}{\leqslant}
\newcommand{\gr}{\geqslant}
\renewcommand{\epsilon}{\varepsilon}
\newcommand{\prend}{$\quad \hfill \Box$}
\usepackage{color}

\usepackage{hyperref}

\begin{document}
\small

\title{\bf On the average volume of sections of convex bodies}

\medskip

\author{S.\ Brazitikos, S.\ Dann, A.\ Giannopoulos and A.\ Koldobsky}

\date{}

\maketitle

\begin{abstract}
\footnotesize The average section functional ${\rm as}(K)$ of a centered convex body in ${\mathbb R}^n$ is the average volume of
the central hyperplane sections of $K$:
\begin{equation*}{\rm as}(K)=\int_{S^{n-1}}|K\cap \xi^{\perp }|\,d\sigma (\xi ).\end{equation*}
We study the question if there exists an absolute constant $C>0$ such that for every $n$, for every centered convex body $K$ in ${\mathbb R}^n$ and
for every $1\ls k\ls n-1$,
\begin{equation*}{\rm as}(K)\ls C^k|K|^{\frac{k}{n}}\,\max_{E\in {\rm Gr}_{n-k}}{\rm as}(K\cap E).\end{equation*}
We observe that the case $k=1$ is equivalent to the hyperplane conjecture. We show that this inequality holds true
in full generality if one replaces $C$ by $CL_K$ or $Cd_{\rm ovr}(K,{\cal{BP}}_k^n)$,
where $L_K$ is the isotropic constant of $K$ and $d_{\rm ovr}(K,{\cal{BP}}_k^n)$ is the outer volume ratio distance from $K$ to
the class ${\cal{BP}}_k^n$ of generalized $k$-intersection bodies. We also compare ${\rm as}(K)$
to the average of ${\rm as}(K\cap E)$ over all $k$-codimensional sections of $K$. We examine separately the dependence of the constants
on the dimension in the case where $K$ is in some of the classical positions as well as the natural lower dimensional analogue
of the average section functional.
\end{abstract}

%%%%%%%%%%%%%%%%%%%%%%%%%%%%%%%%%%%%%%%%%%%%%%%%%%%%%%%%%%%%%%%%%%%%%%%%%%%%%%%%%%%%%%%%%%%%%%%%%%%%%%%%%%%%%%%%%%%%%%%%%%%%%%%%%%%%%%%%%%%%%%%%%%
\section{Introduction}
%%%%%%%%%%%%%%%%%%%%%%%%%%%%%%%%%%%%%%%%%%%%%%%%%%%%%%%%%%%%%%%%%%%%%%%%%%%%%%%%%%%%%%%%%%%%%%%%%%%%%%%%%%%%%%%%%%%%%%%%%%%%%%%%%%%%%%%%%%%%%%%%%%

Let $K$ be a convex body in ${\mathbb R}^n$, with barycenter at the origin (we call these bodies centered).
We denote by ${\rm as}(K)$ the average volume of the central hyperplane sections of $K$:
\begin{equation}\label{eq:intro-1}{\rm as}(K)=\int_{S^{n-1}}|K\cap \xi^{\perp }|\,d\sigma (\xi ),\end{equation}
where $|\cdot |$ denotes volume in the appropriate dimension, $\xi^{\perp }$ is the subspace perpendicular to $\xi $, and $\sigma $
is the rotationally invariant probability measure on $S^{n-1}$. More generally, for any $1\ls r\ls n-1$ we define
\begin{equation}\label{eq:intro-2}{\rm as}_r(K)=\int_{{\rm Gr}_{n-r}}|K\cap E|\,d\nu_{n-r}(E),\end{equation}
where $\nu_{n-r}$ is the Haar probability measure on the Grassmannian ${\rm Gr}_{n-r}$ of $(n-r)$-dimensional subspaces
of ${\mathbb R}^n$. Thus, ${\rm as}_r(K)$ is the average volume of $r$-codimensional central sections of $K$; note that
${\rm as}(K)={\rm as}_1(K)$.

The fourth named author proved in \cite{Koldobsky-2013} that if $K$ is an intersection body in ${\mathbb R}^n$ (see Section \ref{notation}
for definitions and background information) then
\begin{equation}\label{eq:intro-3}{\rm as}(K)\ls b_{n,1}\,|K|^{\frac{1}{n}}\max_{\xi\in S^{n-1}}{\rm as}(K\cap\xi^{\perp }),\end{equation}
where
$$b_{n,1}:=\frac{\omega_{n-1}}{\omega_{n-2}\omega_n^{\frac{1}{n}}}\simeq 1$$
(and $\omega_m$ denotes the volume of the Euclidean unit ball $B_2^m$ in ${\mathbb R}^m$). Whenever we write $a\lesssim b$ we
mean that there exists an absolute constant $c>0$ such that $a\ls cb$, and whenever we write $a\simeq b$, we mean that $a\lesssim b$ and $b\lesssim a$.
Note that \eqref{eq:intro-3} is sharp: it becomes equality if $K=B_2^n$.

The purpose of this article is to discuss similar inequalities for the average volume of hyperplane sections of an arbitrary centered convex body $K$ in ${\mathbb R}^n$. More precisely, we study the following question.

\begin{question}\label{question-intro}Let $1\ls k<n$ and define $\gamma_{n,k}$ as the smallest constant $\gamma >0$ for which the following holds true:
for every centered convex body $K$ in ${\mathbb R}^n$ we have
\begin{equation}\label{eq:intro-4}{\rm as}(K)\ls \gamma^k|K|^{\frac{k}{n}}\,\max_{E\in {\rm Gr}_{n-k}}{\rm as}(K\cap E).\end{equation}
Is it true that $\sup_{n,k}\gamma_{n,k}<\infty $?
\end{question}

In Section \ref{ovr} we generalize \eqref{eq:intro-3} using as a parameter the outer volume ratio distance
$d_{\rm ovr}(K,{\cal{BP}}_k^n)$ from an origin-symmetric convex body $K$ to the class ${\cal{BP}}_k^n$ of generalized
$k$-intersection bodies. Our estimates are based on the next more general theorem which is valid for the larger class of origin-symmetric star bodies 
in ${\mathbb R}^n$ and for any even continuous density on $\R^n$.

\begin{theorem}\label{th:intro-2}Let $1\ls k\ls n-1$, let $K$ be an origin-symmetric star body in $\R^n$, and let $f$ be
a non-negative even continuous function on $\R^n.$ Then
\begin{equation}\label{eq:intro-5} \int_{S^{n-1}} \rho_K^{n-1}(\theta) f(\rho_K(\theta)\theta)\ d\theta
\ls c_{n,k}^k \ d_{\rm ovr}^k (K,{\cal{BP}}_k^n)\ |K|^{\frac kn}
\max_{E\in {\rm Gr}_{n-k}} \int_{S^{n-1}\cap E} \rho_K^{n-k-1}(\theta)f(\rho_K(\theta)\theta)\ d\theta.
\end{equation}
\end{theorem}

In the statement above, $\rho_K$ is the radial function of a star body $K$ and we use the notation $d\theta$ for the non-normalized 
measure on the sphere with density $1$. The constant $c_{n,k}$ is given by
$$c_{n,k}^k=\frac{n\omega_n^{\frac{n-k}n}}{(n-k)\omega_{n-k}},$$
and one can check that $c_{n,k}\simeq 1.$

Theorem \ref{th:intro-2} provides a first estimate on the constants $\gamma_{n,k}$ of Question \ref{question-intro}. Choosing $f\equiv 1$
we see that \eqref{eq:intro-5} implies the following.

\begin{theorem}\label{th:intro-3}Let $1\ls k\ls n-1$, and let $K$ be an origin-symmetric star body in ${\mathbb R}^n$.
Then,
\begin{equation}\label{eq:intro-6}{\rm as}(K)\ls b_{n,k}^kd_{\rm ovr}^k(K,{\cal{BP}}_k^n)\, |K|^{\frac{k}{n}}\,
\max_{E\in Gr_{n-k}} {\rm as}(K\cap E).\end{equation}
In other words, $\gamma_{n,k}\ls b_{n,k}d_{\rm ovr}(K,{\cal{BP}}_k^n)$.\end{theorem}

The constant $b_{n,k}$ in Theorem \ref{th:intro-3} is given by
$$b_{n,k}^k =\frac{\omega_{n-1}}{\omega_{n-k-1} \omega_n^{\frac{k}{n}}},$$
and one can also check that $b_{n,k}\simeq 1$.

\smallskip 

In the case where the body $K$ is convex, the distance $d_{\rm ovr}(K,{\cal{BP}}_k^n)$ was estimated in \cite{KPZ}. In particular, the available bounds for
$d_{\rm ovr}(K,{\cal{BP}}_k^n)$ show that $\gamma_{n,k}$ is bounded by a function of $n/k$, and hence it remains bounded as long as $k$ is
proportional to $n$. More generally, we have:
 
\begin{theorem} \label{th:intro-4}For every origin-symmetric convex body $K$, for every $1\ls k \ls n-1$
and every even non-negative continuous function $f$ on $\R^n$,
\begin{equation}\label{eq:intro-7}\int_{S^{n-1}} \rho_K^{n-1}(\theta) f(\rho_K(\theta)\theta)\ d\theta
\ls \big ( c_1h(n/k)\big )^k \ |K|^{\frac{k}{n}}  \max_{E\in {\rm Gr}_{n-k}} \int_{S^{n-1}\cap E} \rho_K^{n-k-1}(\theta)f(\rho_K(\theta)\theta)\ d\theta,
\end{equation}
where $c_1>0$ is an absolute constant and $h(t)=\sqrt{t}\cdot (\log (et))^{\frac{3}{2}}$, $t\gr 1$. In particular,
\begin{equation}\label{eq:intro-8}\gamma_{n,k}\ls c_1\sqrt{n/k}\,[\log (en/k)]^{\frac{3}{2}}.\end{equation}
\end{theorem}

It is also known that for many classes of convex bodies the distance $d_{\rm ovr}(K,{\cal{BP}}_k^n)$ is bounded by an absolute
constant. This includes unconditional bodies, unit balls of subspaces of $L_p,$ and others. Therefore, the restriction of
Question \ref{question-intro} to all these classes has an affirmative answer.

\smallskip

Theorem \ref{th:intro-2} also allows us to prove an analogue of Theorem \ref{th:intro-3} for the quantities ${\rm as}_r(K).$

\begin{theorem}\label{th:intro-5}Let $1\ls k < n-2$ and $1\ls r < n-k$. For any origin-symmetric star body $K$ in $\R^n$ we have
that
\begin{equation}\label{eq:intro-9}{\rm as}_r(K)\ls \phi_{n,k,r}^k d_{\rm ovr}^k(K,{\cal{BP}}_k^n)\, |K|^{\frac{k}{n}}\,
\max_{E\in Gr_{n-k}} {\rm as}_r(K\cap E).\end{equation}
\end{theorem}

Here
$$\phi_{n,k,r}^k =\frac{\omega_{n-r}}{\omega_{n-k-r}\omega_n^{\frac kn}},$$
and one can check that $\phi_{n,k,r}\simeq \sqrt{\frac{n}{n-r}}$.

\smallskip 

In Section \ref{iso-bounds} we show that an analogue of \eqref{eq:intro-3} holds true in full generality, up to the value of the isotropic constant of $K$.
In order to state our main result we recall the definition of the isotropic position. A centered convex body $K$ of volume $1$ in ${\mathbb R}^n$ is called isotropic if there exists a constant $L_K>0$ such that
\begin{equation}\label{eq:intro-10}\int_K\langle x,\xi\rangle^2dx =L_K^2\end{equation}
for every $\xi\in S^{n-1}$. Every centered convex body $K$ has an {\sl isotropic position} $T(K)$, $T\in GL(n)$, which is uniquely defined
modulo orthogonal transformations, and hence the {\sl isotropic constant} $L_K$ is an invariant of the linear class of $K$. A well-known
question in asymptotic convex geometry asks if there exists an absolute constant $C>0$ such that $L_K\ls C$ for every $n$ and every centered convex
body $K$ in ${\mathbb R}^n$. The best known upper bound
\begin{equation}\label{eq:intro-11}L_n:=\sup\{L_K:K\;\hbox{isotropic in}\;{\mathbb R}^n\}\ls c\sqrt[4]{n}\end{equation}
is due to Klartag \cite{Klartag-2006} (see also \cite{BGVV-book} for the history of the problem and recent developments in this area). On the other hand,
one always has $L_K\gr L_{B_2^n}\gr c$, where $c>0$ is an absolute constant. In other words, the question is if $L_K\simeq 1$ for all centered
convex bodies.

\begin{theorem}\label{th:intro-6}Let $K$ be a centered convex body in ${\mathbb R}^n$. Then, for every $1\ls k\ls n-1$,
\begin{equation}\label{eq:intro-12}{\rm as}(K)\ls (c_2L_K)^k\,|K|^{\frac{k}{n}}\,\max_{E\in {\rm Gr}_{n-k}}{\rm as}(K\cap E),\end{equation}
where $c_2>0$ is an absolute constant and $L_K$ is the isotropic constant of $K$.
\end{theorem}

It is known that for many classes of convex bodies the isotropic constant $L_K$ is bounded by an absolute
constant (see \cite[Chapter 4]{BGVV-book}). Theorem \ref{th:intro-6} provides an affirmative answer to Question \ref{question-intro}
for all these classes.

On the other hand, it is interesting to note that \eqref{eq:intro-12} is essentially the best bound we can hope for.
We show in Proposition \ref{prop:isotropic} that if $K$ is an isotropic convex body in ${\mathbb R}^n$ then
\begin{equation}\label{eq:intro-13}{\rm as}(K)\simeq L_K\,\max_{\xi\in S^{n-1}}{\rm as}(K\cap\xi^{\perp })\,|K|^{\frac{1}{n}}.\end{equation}
This shows that the estimate of Theorem \ref{th:intro-6} is asymptotically sharp: if $\gamma >0$ is a constant such that \eqref{eq:intro-4}
holds for $k=1$ and all $K$ then we must have $\gamma \gr cL_K$. Combining this fact with Theorem \ref{th:intro-6} we actually conclude
that
\begin{equation}\label{eq:intro-14}\gamma_{n,k}\lesssim \gamma_{n,1}\simeq L_n\end{equation}
for all $1\ls k\ls n-1$ (see Proposition \ref{prop:worse-k}).

One of the tools that are used in the proof of Theorem \ref{th:intro-6} is a variant of Meyer's dual Loomis-Whitney inequality \cite{Meyer-1988}
that was recently obtained in \cite{Brazitikos-Giannopoulos-Liakopoulos-2016}; see \eqref{eq:local-2}. The second tool
is a lower bound for the dual affine quermassintegrals
\begin{equation}\label{eq:intro-15}\tilde{\Phi}_k(K):=\frac{\omega_n}{\omega_{n-k}}\left (\int_{{\rm Gr}_{n-k}}|K\cap E|^n\,d\nu_{n-k}(E)\right )^{\frac{1}{n}}
\end{equation}of a convex body $K$ in ${\mathbb R}^n$ in terms of the isotropic constant of $K$ (see \cite{Dafnis-Paouris-2012}). In fact, one can check that the problem to obtain
asymptotically sharp lower bounds for $\tilde{\Phi }_k(K)$ is equivalent to the question whether $\gamma_{n,1}\simeq L_n\simeq 1$ (see Remark \ref{rem:dual-affine}).
When the codimension $k$ is proportional to $n$ the available lower bounds are independent from the isotropic constant of $K$ (see \cite{Dafnis-Paouris-2012} and
\cite[Section~6.4]{BGVV-book}). Thus, we get a variant of Theorem \ref{th:intro-4}.

\begin{theorem}\label{th:intro-7} Let $1\ls k \ls n-1$ and let $K$ be a centered convex body in $\R^n.$ Then,
\begin{equation}\label{eq:intro-16} {\rm as}(K)\ls \big (c_3h(n/k)\big )^k \,|K|^{\frac{k}{n}}\,\max_{E\in Gr_{n-k}} {\rm as}(K\cap E),\end{equation}
where $c_2>0$ is an absolute constant and $h(t)=\sqrt{t}\cdot (\log (et))^{\frac{3}{2}}$, $t\gr 1$.
\end{theorem}

The methods that are used for the proof of Theorem \ref{th:intro-6} and Theorem \ref{th:intro-3} are independent. Note that the
first method allows us to work with (not necessarily symmetric) centered convex bodies while the second method allows us to work
with origin-symmetric (not necessarily convex) star bodies and to consider even continuous densities in place of volume. 
Therefore, the two results complement each other. A link between the two bounds is given by the inequality
\begin{equation}\label{eq:link}L_K\ls cL_k\cdot d_{\rm ovr}(K,{\cal{BP}}_k^n)\end{equation}
which is due to E.~Milman (see \cite[Corollary~5.4]{Milman-2006}). However, since we only know that $L_k=O(\sqrt[4]{k})$,
the estimates of Theorem \ref{th:intro-6} and Theorem \ref{th:intro-3} are incomparable for $k\gg 1$.

\smallskip 

In Section \ref{positions} we discuss the mean value of the average section functional ${\rm as}(K\cap E)$ over all $E\in {\rm Gr}_{n-k}$, $1\ls k\ls n-1$. We obtain
the next general upper and lower bounds.

\begin{theorem}\label{th:intro-8}Let $K$ be a centered convex body in ${\mathbb R}^n$ and define $p(K):=R(K)/|K|^{\frac{1}{n}}$,
where $R(K)$ is the circumradius of $K$. Then, for every $1\ls k\ls n-1$ we have that
\begin{equation}\label{eq:intro-17}\left (\frac{c_4\sqrt{n}}{p(K)}\right )^k\,{\rm as}(K)\ls |K|^{\frac{k}{n}}\,\int_{{\rm Gr}_{n-k}}{\rm as}(K\cap E)\,d\nu_{n-k}(E)\ls \left (\frac{c_5p(K)}{\sqrt{n}}\right )^{\frac{k}{n-1}}\,{\rm as}(K),\end{equation}
where $c_4, c_5>0$ are absolute constants.
\end{theorem}

Since $R(K)$ is polynomial in $n$ for all the classical positions of a convex body $K$ (isotropic position, minimal surface position, minimal
mean width position, John's and L\"{o}wner's position) the right hand side inequality of \eqref{eq:intro-17} implies the following.

\begin{theorem}\label{th:intro-9}Let $K$ be a centered convex body in ${\mathbb R}^n$. If $K$ is in one of the classical positions
then
\begin{equation}\label{eq:intro-18}|K|^{\frac{k}{n}}\,\int_{{\rm Gr}_{n-k}}{\rm as}(K\cap E)\,d\nu_{n-k}(E)\ls c_6^k\,{\rm as}(K)\end{equation}
for every $1\ls k\ls n-1$, where $c_6>0$ is an absolute constant.
\end{theorem}

Closing this introductory section we would like to note that the results of this article are dual to the ones in \cite{Giannopoulos-Koldobsky-Valettas-2016}.
In that work, the main question was to compare the surface area $S(K)$ of a convex body $K$ in ${\mathbb R}^n$
to the minimal, average or maximal surface area of its hyperplane or lower dimensional projections. One of the
main results in \cite{Giannopoulos-Koldobsky-Valettas-2016} states that there exists an absolute constant $c_1>0$ such that, for every convex body $K$ in ${\mathbb R}^n$,
\begin{equation}\label{eq:intro-19}|K|^{\frac{1}{n}}\,\min_{\xi\in S^{n-1}}S(P_{\xi^{\perp }}(K))\ls\frac{c_7\partial_K}{\sqrt{n}}\,S(K),\end{equation}
where $c_7>0$ is an absolute constant and
\begin{equation}\label{eq:intro-20}\partial_K:=\min\Big\{ S(T(K))/|T(K)|^{\frac{n-1}{n}}:T\in GL(n)\Big\}\end{equation}
is the {\sl minimal surface area parameter} of $K$. Another result from \cite{Giannopoulos-Koldobsky-Valettas-2016}
asserts that if $K$ is in some of the classical positions mentioned above, then
\begin{equation}\label{eq:intro-21}|K|^{\frac{1}{n}}\,\int_{S^{n-1}}S(P_{\xi^{\perp }}(K))\,d\sigma (\xi )\gr c_8S(K),\end{equation}
where $c_9>0$ is an absolute constant. The analogy with Theorem \ref{th:intro-6} and Theorem \ref{th:intro-9} is clear; the role of the
average section functional ${\rm as}(K)$ is played by the surface area $S(K)$, and the role of the isotropic constant is played
by the minimal surface area parameter.

%%%%%%%%%%%%%%%%%%%%%%%%%%%%%%%%%%%%%%%%%%%%%%%%%%%%%%%%%%%%%%%%%%%%%%%%%%%%%%%%%%%%%%%%%%%%%%%%%%%%%%%%%%%%%%%%%%%%%%%%%%%%%%%%%%%%%%%%%%%%%%%%%%%%
\section{Notation and background}\label{notation}
%%%%%%%%%%%%%%%%%%%%%%%%%%%%%%%%%%%%%%%%%%%%%%%%%%%%%%%%%%%%%%%%%%%%%%%%%%%%%%%%%%%%%%%%%%%%%%%%%%%%%%%%%%%%%%%%%%%%%%%%%%%%%%%%%%%%%%%%%%%%%%%%%%%%

We work in ${\mathbb R}^n$, which is equipped with a Euclidean structure $\langle\cdot ,\cdot\rangle $. We denote by $\|\cdot \|_2$
the corresponding Euclidean norm, and write $B_2^n$ for the Euclidean unit ball and $S^{n-1}$ for the unit sphere.
Volume is denoted by $|\cdot |$. We write $\omega_n$ for the volume of $B_2^n$ and $\sigma $ for the rotationally invariant probability
measure on $S^{n-1}$. We use the notation $d\theta$ for the non-normalized measure on the sphere with density $1$.

If $\xi\in S^{n-1}$, then $\xi^{\perp }=\{x\in {\mathbb R}^n:\langle x,\xi\rangle =0\}$.
The Grassmann manifold ${\rm Gr}_m$ of $m$-dimensional subspaces of ${\mathbb R}^n$ is equipped with the Haar probability
measure $\nu_{m}$. For every $1\ls m\ls n-1$ and $E\in {\rm Gr}_m$ we write $P_E$ for the orthogonal projection from $\mathbb R^{n}$ onto $E$,
and we set $B_E=B_2^n\cap F$ and $S_E=S^{n-1}\cap E$. The letters $c,c^{\prime }, c_1, c_2$ etc. denote absolute positive constants which may change from line to line.

We refer to the books \cite{Gardner-book} and \cite{Schneider-book} for basic facts from the Brunn-Minkowski theory and to the book
\cite{AGA-book} for basic facts from asymptotic convex geometry.

\medskip

%%%%%%%%%%%%%%%%%%%%%%%%%%%%%%%%%%%%%%%%%%%%%%%%%%%%%%%%%%%%%%%%%%%%%%%%%%%%%%%%%%%%%%%%%%%%%%%%%%%%%%%%%%%%%%%%%%%%%%%%%%%%%%%%%%%%%%%%%%
\noindent {\bf 2.1. Star bodies and convex bodies}. A convex body in ${\mathbb R}^n$ is a compact convex
subset $K$ of ${\mathbb R}^n$ with non-empty interior. We say that $K$ is
symmetric if $x\in K$ implies that $-x\in K$, and that $K$ is centered if its barycenter $\frac{1}{|K|}\int_Kx\,dx $ is at the origin.
A compact set $K$ in ${\mathbb R}^n$ will be called star-shaped at $0$ if it contains the origin in its interior and every line through
$0$ meets $K$ in a line segment. For such a set, the radial function $\rho_K $ is defined on $S^{n-1}$ by
\begin{equation}\label{eq:notation-1}\rho_K(\theta )=\max \{\lambda >0:\lambda\theta\in K\},\qquad \theta\in S^{n-1}.\end{equation}
If $\rho_K $ is continuous, then we say that $K$ is a star body. Then, the volume of $K$ in polar coordinates
is given by
\begin{equation}\label{eq:notation-2}|K|=\omega_n\int_{S^{n-1}}\rho_K^n(\theta )\,d\sigma (\theta ).\end{equation}
The radial sum $K\tilde{+} D$ of two star bodies $K$ and $D$ is defined by
\begin{equation}\label{eq:notation-3}\rho_{K\tilde{+} D}= \rho_K + \rho_D.\end{equation}
We equip the class ${\cal S}_n$ of star bodies with the radial metric
\begin{equation}\label{eq:notation-4}d_r(K,D):=\sup_{\xi\in S^{n-1}} |\rho_K(\xi)-\rho_D(\xi)|.\end{equation}
The support function of a convex body $K$ is defined by $h_K(y)=\max \{\langle x,y\rangle :x\in K\}$, and the mean width of $K$ is
\begin{equation}\label{eq:notation-5}w(K)=\int_{S^{n-1}}h_K(\theta )\,d\sigma (\theta ). \end{equation}
The circumradius of $K$ is the smallest $R>0$ for which $K\subseteq RB_2^n$.
If $0\in {\rm int}(K)$ then we write $r(K)$ for the inradius of $K$ (the largest $r>0$
for which $rB_2^n\subseteq K$) and we define the polar body $K^{\circ }$ of $K$ by
\begin{equation}\label{eq:notation-6}K^{\circ }:=\{ y\in {\mathbb R}^n: \langle x,y\rangle \ls 1 \;\hbox{for all}\; x\in K\}. \end{equation}
The section of a star body $K$ with $\xi^{\perp }$ is denoted by $K\cap\xi^{\perp }$,
and we write $P_{\xi^{\perp } }(K)$ for the orthogonal projection of $K$ onto $\xi^{\perp }$.

The volume radius of $K$ is the quantity ${\rm vrad}(K)=\left (|K|/|B_2^n|\right )^{1/n}$.
We also define $\|\theta \|_K=\min\{ t>0:\theta \in tK\}$ and
\begin{equation}\label{eq:notation-7}M(K)=\int_{S^{n-1}}\rho_K^{-1}(\theta )\,d\sigma (\theta )=\int_{S^{n-1}}\|\theta\|_K\,d\sigma (\theta ).\end{equation}

\medskip

\noindent {\bf 2.2. Dual mixed volumes}. Lutwak introduced dual mixed volumes in \cite{Lutwak-1975}; he
first considered convex bodies, but then extended his definition to the class ${\mathcal S}_n$ of star
bodies. Given $K_1,\ldots,K_n\in {\mathcal S}_n$, their dual mixed volume is the integral
\begin{equation}\label{eq:notation-8}\tilde{V}(K_1,\ldots,K_n)=\omega_n \int_{S^{n-1}}
\rho_{K_1}(\theta )\cdots \rho_{K_n}(\theta )d\sigma (\theta ).\end{equation}
The observation is that such integrals have properties analogous to those of mixed volumes if one replaces Minkowski addition by
radial addition. The function $\tilde{V}$ is clearly non-negative, symmetric and monotone with respect to its arguments,
positive linear with respect to $\tilde{+}$ in each of its arguments, and has volume as its diagonal. A simple calculation shows that if
$K_1,\ldots,K_m\in {\mathcal S}_n$ and $\lambda_1,\ldots,\lambda_m>0$, then
\begin{equation}\label{eq:notation-9}|\lambda_1 K_1 \tilde{+}\cdots\tilde{+}\lambda_mK_m|=
\sum_{i_1,\ldots,i_n=1}^m \tilde{V}(K_{i_1},\ldots,K_{i_n})\lambda_{i_1}\ldots\lambda_{i_n}.\end{equation}
In particular, if $K,D\in {\mathcal S}_n$ and $t>0$ then
\begin{equation}\label{eq:notation-10}|K\tilde{+}tD|=\sum_{j=0}^n {n\choose j}\tilde{V}_j(K,D)\;t^j,\end{equation}
where $\tilde{V}_j(K,D):=\omega_n\int_{S^{n-1}}\rho_K^{n-j}(\theta )\rho_D^{j}(\theta )d\sigma (\theta )$, is the
$j$-th dual mixed volume of $K$ and $D$.

An inequality which further illustrates the analogy with the mixed volumes is the dual Minkowski's
inequality: for every $K,D\in {\mathcal S}_n$ an application of H\"older's inequality gives
\begin{equation}\label{eq:notation-11}\tilde{V}_1(K,D) \ls \left( \omega_n\int_{S^{n-1}}\rho_K^n(\theta )
d\sigma (\theta )\right)^{\frac{n-1}{n}}\left(\omega_n\int_{S^{n-1}}\rho_D^n(\theta )d\sigma (\theta )\right)^{\frac{1}{n}}
\ls |K|^{\frac{n-1}{n}}|D|^{\frac{1}{n}}.\end{equation}

\medskip

\noindent {\bf 2.3. Intersection bodies}. The class of intersection bodies was introduced by Lutwak \cite{Lutwak-1988}.
The intersection body of a star body $K$ in ${\mathbb R}^n$
with $\rho_K\in C(S^{n-1})$ is the star body $IK$ with radial function
\begin{equation}\label{eq:notation-12}\rho_{IK}(\xi )=|K\cap \xi^{\perp }|=\omega_{n-1}\int_{S(\xi^{\perp })}\rho_K^{n-1}(\theta )d\sigma_{\xi }(\theta ),\end{equation}
where $S(\xi^{\perp })=S^{n-1}\cap\xi^{\perp }$ is the Euclidean unit sphere of $\xi^{\perp }$ and
$\sigma_{\xi }$ denotes the rotationally invariant probability measure on $S(\xi^{\perp })$.
If $K$ is a centered convex body then $IK$ is a symmetric convex body. It is known that
\begin{equation}\label{eq:notation-13}I(TK)=|\det T|\,(T^{-1})^{\ast }(IK)\end{equation}
for every $T\in GL(n)$. In particular, if $T\in SL(n)$ we see that $|I(TK)|=|IK|$.
The class of intersection bodies ${\cal{I}}_n$ is defined as the closure in the radial metric
of intersection bodies of star bodies.

Zhang introduced more general classes of bodies in \cite{Zhang-1996}.
For $1\ls k \ls n-1,$  the {\it $(n-k)$-dimensional spherical Radon transform}
$R_{n-k}:C(S^{n-1})\to C({\rm Gr}_{n-k})$
is a linear operator defined by
\begin{equation}\label{eq:notation-14}R_{n-k}g (E)=\int_{S^{n-1}\cap E} g(\theta)\ d\theta,\qquad E\in {\rm Gr}_{n-k}\end{equation}
for every function $g\in C(S^{n-1}).$ We say that
an origin-symmetric star body $D$ in $\R^n$ is a {\it generalized $k$-intersection body},
and write $D\in {\cal{BP}}_k^n$, if there exists a finite non-negative Borel measure $\mu_D$
on ${\rm Gr}_{n-k}$ so that for every $g\in C(S^{n-1})$
\begin{equation}\label{eq:notation-16}
\int_{S^{n-1}} \rho_D^{k}(\theta) g(\theta)\ d\theta=\int_{{\rm Gr}_{n-k}} R_{n-k}g(H)\ d\mu_D(H).
\end{equation}
For a star body $K$ in $\R^n$ and $1\ls k \ls n-1$, we denote by
$$d_{\rm {ovr}}(K,{\cal{BP}}_k^n) = \inf \left\{ \left( \frac {|D|}{|K|}\right)^{1/n}:\ K\subset D,\ D\in {\cal{BP}}_k^n \right\}$$
the outer volume ratio distance from $K$ to the class ${\cal{BP}}_k^n$. The reader will find more
information on the Radon transform and intersection bodies in the book \cite{Koldobsky-book}.

%%%%%%%%%%%%%%%%%%%%%%%%%%%%%%%%%%%%%%%%%%%%%%%%%%%%%%%%%%%%%%%%%%%%%%%%%%%%%%%%%%%%%%%%%%%%%%%%%%%%%%%%%%%%%%%%%%%%%%%%%%%%%%%%%%%%%%%%%%%%%%%%%%%%
\section{Bounds in terms of the outer volume ratio distance to the class of generalized $k$-intersection bodies} \label{ovr}
%%%%%%%%%%%%%%%%%%%%%%%%%%%%%%%%%%%%%%%%%%%%%%%%%%%%%%%%%%%%%%%%%%%%%%%%%%%%%%%%%%%%%%%%%%%%%%%%%%%%%%%%%%%%%%%%%%%%%%%%%%%%%%%%%%%%%%%%%%%%%%%%%%%%

The main result of this section is Theorem \ref{th:intro-2} which is valid for the larger class of origin-symmetric star bodies
in ${\mathbb R}^n$ and for any even continuous density on $\R^n$. By an appropriate choice of the density $f$ we obtain Theorem \ref{th:intro-3} and
its generalization in Theorem \ref{th:intro-5}.

\medskip 

\noindent{\bf Proof of Theorem \ref{th:intro-2}.} Let $\e>0.$ For every $E\in {\rm Gr}_{n-k},$ we have
\begin{equation}\label{eq:ovr-1}\int_{(K\tilde{+} \e B_2^n)\cap E} f(x) dx -  \int_{K\cap E} f(x) dx\ls  \max_{F\in {\rm Gr}_{n-k}} \left (\int_{(K\tilde{+} \e B_2^n)\cap F} f(x) dx -  \int_{K\cap F} f(x) dx\right ).\end{equation}
Note that $\rho_{K\tilde{+}\e B_2^n}=\rho_K+\e$. Expressing the integrals in polar coordinates we get
\begin{equation}\label{eq:ovr-2}R_{n-k}\left(\int_{\rho_K(\cdot)}^{\rho_{K}(\cdot)+\e}
r^{n-k-1}f(r\cdot) dr \right)(E) \ls \max_{F\in {\rm Gr}_{n-k}} \left(\int_{S^{n-1}\cap F}
\int_{\rho_K(\theta)}^{\rho_{K}(\theta)+\e} r^{n-k-1}f(r\theta) dr d\theta \right).
\end{equation}
Let $D\in {\cal{BP}}_k^n$ such that $K\subset D.$ Integrating the latter inequality by $E$ over ${\rm Gr}_{n-k}$ with the measure $\mu_D$
corresponding to $D$ by \eqref{eq:notation-14}, we get
\begin{align}\label{eq:ovr-3}&\int_{S^{n-1}}\rho_D^k(\theta)\int_{\rho_K(\theta)}^{\rho_{K}(\theta)+\e}
r^{n-k-1}f(r\theta) dr d\theta\\
\nonumber &\hspace*{1.5cm}  \ls \mu_D({\rm Gr}_{n-k})\cdot 
\max_{F\in {\rm Gr}_{n-k}} \left(\int_{S^{n-1}\cap F} \int_{\rho_K(\theta)}^{\rho_{K}(\theta)+\e} r^{n-k-1}f(r\theta) dr d\theta \right).
\end{align}
We divide both sides by $\e$ and send $\e$ to zero. Note that we can interchange the limit with the maximum,
because the convergence is uniform with respect to $F$. Thus, we get
\begin{equation}\label{eq:ovr-4}\int_{S^{n-1}}\rho_D^k(\theta)\rho_K^{n-k-1}(\theta) f(\rho_K(\theta)\theta)d\theta\ls \mu_D({\rm Gr}_{n-k})\cdot \max_{F\in {\rm Gr}_{n-k}} \left(\int_{S^{n-1}\cap F} \rho_K^{n-k-1}(\theta)
f(\rho_K(\theta)\theta) d\theta\right).\end{equation}
The integral in the left hand side can be estimated from below by $\int_{S^{n-1}} \rho_K^{n-1}(\theta)
f(\rho_K(\theta)\theta)\ d\theta,$
because $K\subset D.$

To estimate $\mu_D({\rm Gr}_{n-k})$ from above, we combine the fact that $1= R_{n-k}{\bf 1}(E)/|S^{n-k-1}|$ for every $E\in {\rm Gr}_{n-k}$ with
Definition \eqref{eq:notation-14} and H\"older's inequality to write
\begin{align}\label{eq:ovr-5} \mu_D({\rm Gr}_{n-k}) &= \frac 1{\left|S^{n-k-1}\right|} \int_{{\rm Gr}_{n-k}} R_{n-k}{\bf 1}(E) d\mu_D(E)\\
\nonumber &=\frac 1{\left| S^{n-k-1} \right| } \int_{S^{n-1}} \|\theta\|_D^{-k}\ d\theta \\
\nonumber &\ls  \frac 1{\left|S^{n-k-1}\right|} \left|S^{n-1}\right|^{\frac{n-k}n} \left(\int_{S^{n-1}} \|\theta\|_D^{-n}\ d\theta\right)^{\frac kn}\\
\nonumber &=  \frac{1}{\left|S^{n-k-1}\right|} \left|S^{n-1}\right|^{\frac{n-k}n} n^{\frac{k}{n}}|D|^{\frac{k}{n}}.
\end{align}
These estimates show that
\begin{align}\label{eq:ovr-6}&\int_{S^{n-1}} \rho_K^{n-1}(\theta)f(\rho_K(\theta)\theta) d\theta \\
\nonumber &\hspace*{1.5cm} \ls \frac{1}{\left|S^{n-k-1}\right|} \left|S^{n-1}\right|^{\frac{n-k}n} n^{\frac{k}{n}}|D|^{\frac{k}{n}}
\max_{F\in {\rm Gr}_{n-k}} \left(\int_{S^{n-1}\cap F} \rho_K^{n-k-1}(\theta)f(\rho_K(\theta)\theta) d\theta\right).
\end{align}
Finally, we choose $D$ so that $|D|^{1/n}\ls (1+\delta)d_{\rm ovr}(K,{\cal{BP}}_k^n)|K|^{1/n},$
and then send $\delta$ to zero. \qed

\medskip

All the other results of this section are consequences of Theorem \ref{th:intro-2}.

\medskip 

\noindent {\bf Proof of Theorem \ref{th:intro-3}.} First, we express the average section functionals ${\rm as}(K)$ and ${\rm as}(K\cap E)$ 
in terms of the radial function of $K$. Using \eqref{eq:notation-2} we write
\begin{align}\label{eq:ovr-7}{\rm as}(K) &=\int_{S^{n-1}}|K\cap\xi^{\perp }|\,d\sigma (\xi )=\omega_{n-1}\int_{S^{n-1}}
\int_{S(\xi^{\perp })}\rho_K^{n-1}(\theta )\,d\sigma_{\xi }(\theta )\,d\sigma (\xi )\\
\nonumber &=\omega_{n-1}\int_{S^{n-1}}\rho_K^{n-1}(\theta )\,d\sigma (\theta ).\end{align}
Similarly, for every $1\ls k\ls n-1$ and any $E\in {\rm Gr}_{n-k}$, we have
\begin{equation}\label{eq:ovr-8}{\rm as}(K\cap E)=\omega_{n-k-1}\int_{S_E}\rho_K^{n-k-1}(\theta )\,d\sigma_E(\theta ),\end{equation}
where $\sigma_E$ is the rotationally invariant probability measure on $S_E=S^{m-1}\cap E$.
Applying Theorem \ref{th:intro-2} for the density $f\equiv {\bf 1}$ we get
\begin{equation}\label{eq:ovr-9} \int_{S^{n-1}} \rho_K^{n-1}(\theta) \ d\theta
\ls c_{n,k}^k \ d_{\rm ovr}^k (K,{\cal{BP}}_k^n)\ |K|^{\frac kn}
\max_{E\in {\rm Gr}_{n-k}} \int_{S^{n-1}\cap E} \rho_K^{n-k-1}(\theta)\ d\theta ,
\end{equation}
and Theorem \ref{th:intro-3} follows from \eqref{eq:ovr-7} and \eqref{eq:ovr-8} and an adjustment of the constants. \qed 

\begin{remark}\label{rem:ovr}\rm For certain classes of origin-symmetric convex bodies the distance $d_{\rm ovr}(K,{\cal{BP}}_k^n)$ is
bounded by an absolute constant. These classes include unconditional convex bodies and duals of bodies
with bounded volume ratio (see \cite{Koldobsky-2015}) and the unit balls of normed spaces that embed in $L_p,\ -n<p<\infty$
(see \cite{Koldobsky-2016}, \cite{Milman-2006} and \cite{Koldobsky-Pajor}). If we restrict
Question \ref{question-intro} to any of these classes then Theorem \ref{th:intro-3} provides an affirmative answer.
\end{remark}

\noindent {\bf Proof of Theorem \ref{th:intro-4}.} We combine Theorem \ref{th:intro-2} with the following result from \cite{KPZ}: For every origin-symmetric convex body $K$
in $\R^n$,
\begin{equation}\label{eq:ovr-10}d_{\rm ovr}(K,{\cal{BP}}_k^n)\ls c\sqrt{n/k}\,[\log (en/k)]^{\frac{3}{2}},\end{equation}
where $c>0$ is an absolute constant. \qed 

\medskip

\noindent {\bf Proof of Theorem \ref{th:intro-5}.} We choose $f(x)=\|x\|_2^{-r+1}$ in Theorem \ref{th:intro-2} to get
\begin{equation}\label{eq:ovr-11} \int_{S^{n-1}} \rho_K^{n-r}(\theta)\ d\theta
\ls c_{n,k}^k \ d_{\rm ovr}^k (K,{\cal{BP}}_k^n)\ |K|^{\frac kn}
\max_{E\in {\rm Gr}_{n-k}} \int_{S^{n-1}\cap E} \rho_K^{n-k-r}(\theta)\ d\theta.
\end{equation}
Then, we apply the formula
\begin{equation}\label{eq:ovr-12}{\rm as}_r(K)= \omega_{n-r} \int_{S^{n-1}} \rho_K^{n-r}(\theta) d\sigma(\theta)\end{equation}
which generalizes \eqref{eq:ovr-7} and is easily verified in the same way. \qed

%%%%%%%%%%%%%%%%%%%%%%%%%%%%%%%%%%%%%%%%%%%%%%%%%%%%%%%%%%%%%%%%%%%%%%%%%%%%%%%%%%%%%%%%%%%%%%%%%%%%%%%%%%%%%%%%%%%%%%%%%
\section{Bounds in terms of the isotropic constant}\label{iso-bounds}
%%%%%%%%%%%%%%%%%%%%%%%%%%%%%%%%%%%%%%%%%%%%%%%%%%%%%%%%%%%%%%%%%%%%%%%%%%%%%%%%%%%%%%%%%%%%%%%%%%%%%%%%%%%%%%%%%%%%%%%%%

Let $K$ be a centered convex body in ${\mathbb R}^n$. In this section we compare ${\rm as}(K)$ with the corresponding
average section functional ${\rm as}(K\cap E)$ for any $k$-codimensional subspace $E$ of ${\mathbb R}^n$. Our main tool will be a recent result from \cite{Brazitikos-Giannopoulos-Liakopoulos-2016} which is a restricted version of Meyer's dual Loomis-Whitney inequality
\begin{equation}\label{eq:local-1}|K|^{n-1}\gr \frac{n!}{n^n}\prod_{i=1}^n|K\cap e_i^{\perp }|\end{equation}
where $\{e_1,\ldots ,e_n\}$ is any orthonormal basis of ${\mathbb R}^n$ (see \cite{Meyer-1988}) and in a sense dualizes the uniform cover inequality of Bollob\'{a}s and Thomason
(see \cite{Bollobas-Thomason-1995}). In order to give the precise statement, we introduce some notation. For every non-empty
$\tau\subset [n]:=\{1,\ldots ,n\}$ we set $F_{\tau }={\rm span}\{e_j:j\in\tau \}$ and $E_{\tau }=F_{\tau }^{\perp }$.
Given $s\gr 1$ and $\sigma\subseteq [n]$, following the terminology of \cite{Bollobas-Thomason-1995} we say that the (not necessarily distinct) sets $\sigma_1,\ldots ,\sigma_t\subseteq \sigma $ form an $s$-uniform cover of $\sigma $ if every $j\in \sigma $ belongs to exactly $s$ of the sets $\sigma_i$.
Then, \cite[Theorem~1.3]{Brazitikos-Giannopoulos-Liakopoulos-2016} states that for any centered convex body $K$ in ${\mathbb R}^n$,
for any $t\gr 1$ and any $s$-uniform cover $(\sigma_1,\ldots ,\sigma_t)$ of a subset $\sigma $ of $[n]$ we have
\begin{equation}\label{eq:local-2}\prod_{i=1}^t|K\cap E_{\sigma_i}|\ls \left (\frac{c_0t}{s}\right )^{ds}|K\cap E_{\sigma }|^s|K|^{t-s},\end{equation}
where $d=|\sigma |$. We will need a special case of this inequality. We consider $1\ls k\ls n-1$ and a $(k+1)$-tuple of orthonormal vectors
$e_1,\ldots ,e_k,e_{k+1}:=\xi $ in ${\mathbb R}^n$. Note that the sets $\sigma_1=[k]$ and $\sigma_2=\{k+1\}$ form a $1$-uniform
cover of the set $\sigma =[k+1]$. Applying \eqref{eq:local-2} with $t=2$, $s=1$ and $d=k+1$ we obtain the next lemma.

\begin{lemma}\label{lem:local-1}Let $K$ be a centered convex body in ${\mathbb R}^n$. For every $1\ls k\ls n-1$, for any $E\in {\rm Gr}_{n-k}$
and any $\xi\in S^{n-1}\cap E$ we have
\begin{equation}\label{eq:local-3}|K\cap E|\cdot |K\cap\xi^{\perp }|\ls c_0^{k+1}|K\cap E\cap\xi^{\perp }|\cdot |K|,\end{equation}
where $c_0>0$ is an absolute constant.
\end{lemma}

Using Lemma \ref{lem:local-1} we can compare ${\rm as}(K)$ to ${\rm as}(K\cap E)$ for every $E\in {\rm Gr}_{n-k}$. We need the next well-known
properties of the parameter $M$ which was defined by \eqref{eq:notation-7}. If $D$ is a symmetric convex body in ${\mathbb R}^m$ then
for every $1\ls s\ls m-1$ and $F\in {\rm Gr}_s({\mathbb R}^m)$ we have that
\begin{equation}\label{eq:M-1}M(D\cap F)=\int_{S_F}\|\xi \|_D\,d\sigma_F (\xi )\ls c_1\sqrt{m/s}\,\int_{S^{m-1}}\|\xi \|_D\,d\sigma (\xi )
=c_1\sqrt{m/s}\,M(D),\end{equation}
where $c_1>0$ is an absolute constant. It is also known that
\begin{equation}\label{eq:M-2}\int_{S^{m-1}}\rho_D(\theta )\,d\sigma (\theta )=\int_{S^{m-1}}\|\theta \|_D^{-1}\,d\sigma (\theta )\simeq \frac{1}{M(D)}.\end{equation}
For a proof of \eqref{eq:M-1} and \eqref{eq:M-2} see \cite[Section~5.2.1]{AGA-book} and \cite[Theorem~5.8.7]{AGA-book} respectively.

\begin{theorem}\label{th:BGL-1}Let $K$ be a centered convex body in ${\mathbb R}^n$. For every $1\ls k\ls n-1$ and $E\in {\rm Gr}_{n-k}$
we have
\begin{equation}\label{eq:BGL-1}|K\cap E|\cdot {\rm as}(K)\ls  c_2^k\,{\rm as}(K\cap E)\cdot |K|,\end{equation}
where $c_2>0$ is an absolute constant.
\end{theorem}

\begin{proof}We consider an orhonormal basis $\{e_1,\ldots ,e_k\}$ of $E^{\perp }$ and any unit vector $\xi\in E$. From Lemma \ref{lem:local-1} we have
\begin{equation}\label{eq:BGL-2}|K\cap E|\cdot |K\cap\xi^{\perp }|\ls c_0^{k+1}|K\cap E\cap\xi^{\perp }|\cdot |K|.\end{equation}
Integrating \eqref{eq:BGL-2} with respect to $\xi\in S_E$ we see that
\begin{align}\label{eq:BGL-3}|K\cap E|\cdot \int_{S_E}|K\cap \xi^{\perp }|d\sigma_E(\xi )
&\ls c_0^{k+1}\int_{S_E}|(K\cap E)\cap\xi^{\perp }|d\sigma_E(\xi )\cdot |K|\\
\nonumber &= c_0^{k+1}\,{\rm as}(K\cap E)\cdot |K|.\end{align}
Applying \eqref{eq:M-2} for the symmetric convex body $IK\cap E$ we see that
\begin{equation}\label{eq:BGL-4}\int_{S_E}|K\cap\xi^{\perp }|d\sigma_E (\xi )
=\int_{S_E}\rho_{IK}(\xi )\,d\sigma_E (\xi )\simeq \frac{1}{M(IK\cap E)},\end{equation}
and hence, using \eqref{eq:M-1} with $m=n$ and $s=n-k$ and then applying \eqref{eq:M-2} for the body $IK$ this time, we obtain
\begin{align}\int_{S_E}|K\cap\xi^{\perp }|d\sigma_E (\xi )&\gr \frac{c\sqrt{n-k}}{\sqrt{n}}\frac{1}{M(IK)}\simeq \frac{\sqrt{n-k}}{\sqrt{n}}\int_{S^{n-1}}\rho_{IK}(\xi )\,d\sigma (\xi  )\\
\nonumber &= \frac{\sqrt{n-k}}{\sqrt{n}}\int_{S^{n-1}}|K\cap\xi^{\perp }|\,d\sigma (\xi )=\frac{\sqrt{n-k}}{\sqrt{n}}{\rm as}(K).
\end{align}
Therefore,
\begin{equation}\label{eq:BGL-5}|K\cap E|\,{\rm as}(K)\ls \frac{c_1\sqrt{n}}{\sqrt{n-k}}c_0^{k+1}\,{\rm as}(K\cap E)\cdot |K|\ls
c_2^k\,{\rm as}(K\cap E)\cdot |K|\end{equation}
for every $E\in {\rm Gr}_{n-k}$. \end{proof}

\medskip

For the proof of Theorem \ref{th:intro-6} we use Theorem \ref{th:BGL-1} and estimates for the dual affine quermassintegrals of a centered convex body $K$: these
are defined, for any $1\ls k\ls n-1$, as follows:
\begin{equation}\label{eq:main-3}\tilde{R}_k(K):=\frac{1}{|K|^{n-k}}\int_{{\rm Gr}_{n-k}}|K\cap E|^n\,d\nu_{n-k}(E).\end{equation}
The quantities $\tilde{R}_k(K)$ were introduced by Lutwak in \cite{Lutwak-1984} and \cite{Lutwak-1988}. More precisely, he considered
the quantities $\tilde{\Phi}_k(K)$ that were introduced in \eqref{eq:intro-15},
which clearly satisfy the identity
\begin{equation}\tilde{\Phi }_k(K)=\frac{\omega_n}{\omega_{n-k}}|K|^{\frac{n-k}{n}}\big [\tilde{R}_k(K)\big ]^{\frac{1}{n}}.\end{equation}
Grinberg proved in \cite{Grinberg-1990} that the quantity $\tilde{R}_k(K)$ is invariant under $T\in GL(n)$: one has
\begin{equation}\label{eq:main-5}\tilde{R}_k(T(K))=\tilde{R}_k(K)\end{equation}
for every $T\in GL(n)$. He also proved that
\begin{equation}\label{eq:main-4}\tilde{R}_k(K)\ls \tilde{R}_k(B_2^n):=\frac{\omega_{n-k}^n}{\omega_n^{n-k}}\ls e^{\frac{kn}{2}}.\end{equation}
On the other hand, it was observed by Dafnis and Paouris in \cite{Dafnis-Paouris-2012} that
\begin{equation}\label{eq:main-6}\tilde{R}_k(K)\gr \left (\frac{c_4}{L_K}\right )^{kn},\end{equation}
where $c_4>0$ is an absolute constant. We will use this lower bound, which is an immediate consequence of \eqref{eq:main-5}
and of the fact that if $K$ is isotropic then $|K\cap E|^{\frac{1}{k}}\gr\frac{c_4}{L_K}$ for every $E\in {\rm Gr}_{n-k}$ (see \cite[Proposition~5.1.15]{BGVV-book} for a proof).

\medskip

\noindent {\bf Proof of Theorem \ref{th:intro-6}.}  Let $K$ be a centered convex body in ${\mathbb R}^n$ and fix $1\ls k\ls n-1$. From Theorem
\ref{th:BGL-1} we know that for every $E\in {\rm Gr}_{n-k}$ we have
\begin{equation}\label{eq:BGL-15}|K\cap E|\cdot {\rm as}(K)\ls  c_2^k\,{\rm as}(K\cap E)\cdot |K|,\end{equation}
where $c_2>0$ is an absolute constant. Therefore,
\begin{equation}\label{eq:BGL-6}\max_{E\in {\rm Gr}_{n-k}}\,|K\cap E|\cdot {\rm as}(K)\ls c_2^k\,\max_{E\in {\rm Gr}_{n-k}}{\rm as}(K\cap E)\cdot |K|.\end{equation}
Next, from \eqref{eq:main-5} we see that
\begin{equation}\label{eq:BGL-8}\max_{E\in {\rm Gr}_{n-k}}|K\cap E|\gr \left (\int_{{\rm Gr}_{n-k}}|K\cap E|^n\,d\nu_{n-k}(E)\right )^{\frac{1}{n}}\gr \left (\frac{c_4}{L_K}\right )^k\,|K|^{\frac{n-k}{n}}.\end{equation}
Going back to \eqref{eq:BGL-6} we see that
\begin{equation}\label{eq:BGL-9}\left (\frac{c_4}{L_K}\right )^k|K|^{\frac{n-k}{n}}{\rm as}(K)\ls c_2^k|K|\,\max_{E\in {\rm Gr}_{n-k}}{\rm as}(K\cap E),\end{equation}
and this proves Theorem \ref{th:intro-6}. \prend

\medskip

The next proposition shows that if $K$ is isotropic and if we consider the hyperplane case (where $k=1$) then the estimate of
Theorem \ref{th:intro-6} is sharp: we have an asymptotic formula.

\begin{proposition}\label{prop:isotropic}Let $K$ be an isotropic convex body in ${\mathbb R}^n$. Then,
${\rm as}(K)\simeq L_K^{-1}$ and ${\rm as}(K\cap\xi^{\perp })\simeq L_K^{-2}$ for all $\xi\in S^{n-1}$.
In particular,
\begin{equation}\label{eq:isotropic-1}{\rm as}(K)\simeq L_K\,|K|^{\frac{1}{n}}\max_{\xi\in S^{n-1}}{\rm as}(K\cap\xi^{\perp }).\end{equation}
\end{proposition}

\begin{proof}It is a general fact (following from \cite[Proposition~5.1.15]{BGVV-book}) that if $K$ is an isotropic convex body then, for every $E\in {\rm Gr}_{n-k}$ we have
\begin{equation}\label{eq:isotropic-2}\frac{c_1}{L_K}\ls |K\cap E|^{\frac{1}{k}}\ls \frac{cL_k}{L_K}\ls \frac{c_2(k)}{L_K},\end{equation}
where $c_1>0$ is an absolute constant and $c_2(k)$ is a positive constant depending only on $k$ (in fact, $c_2(k)\ls c\sqrt[4]{k}$ by
Klartag's estimate on $L_k$). Applying \eqref{eq:isotropic-2} with $k=1$ we see that all hyperplane sections $K\cap\xi^{\perp }$
of $K$ have volume equal (up to an absolute constant) to $L_K^{-1}$. In particular,
\begin{equation}\label{eq:isotropic-3}{\rm as}(K)=\int_{S^{n-1}}|K\cap \xi^{\perp }|\,d\sigma (\xi )\simeq L_K^{-1}.\end{equation}
Applying \eqref{eq:isotropic-2} with $k=2$ we see that all $2$-codimensional sections $K\cap E$ of $K$ have volume equal
(up to an absolute constant) to $L_K^{-2}$. In particular, for every $\xi\in S^{n-1}$ we get
\begin{equation}\label{eq:isotropic-4}{\rm as}(K\cap \xi^{\perp })=\int_{S(\xi^{\perp })}|K\cap E_{\xi ,\theta }|\,d\sigma_{\xi }(\theta )\simeq
L_K^{-2},\end{equation}
where $E_{\xi ,\theta }=[{\rm span}\{\xi ,\theta\}]^{\perp }$. This shows that
\begin{equation}\label{eq:isotropic-5}{\rm as}(K)\simeq L_K\,{\rm as}(K\cap\xi^{\perp })=L_K\,{\rm as}(K\cap\xi^{\perp })\,|K|^{\frac{1}{n}}\end{equation}
In particular, \eqref{eq:isotropic-5} implies \eqref{eq:isotropic-1}. \end{proof}

\begin{remark}\rm Proposition \ref{prop:isotropic} and the definition of $\gamma_{n,1}$ show that
\begin{equation}\label{eq:worse-k-1}L_K^{-1}\simeq {\rm as}(K)\ls \gamma_{n,1}\,\max_{\xi\in S^{n-1}}{\rm as}(K\cap \xi^{\perp })\simeq \gamma_{n,1}L_K^{-2}\end{equation}
for every isotropic convex body $K$ in ${\mathbb R}^n$. Therefore, $L_K\ls c\gamma_{n,1}$ for some absolute constant $c>0$, which implies
that
\begin{equation}\label{eq:worse-k-2}L_n\ls c\gamma_{n,1}.\end{equation}
Note that by Theorem \ref{th:intro-6} we can then conclude that $\gamma_{n,k}\ls cL_n\ls c\gamma_{n,1}$.
Finally, Theorem \ref{th:intro-6} shows that $\gamma_{n,1}\ls c^{\prime }L_n$. We summarize in the next proposition.
\end{remark}

\begin{proposition}\label{prop:worse-k}For any $1\ls k\ls n-1$ we have
\begin{equation}\label{eq:worse-k-3}\gamma_{n,k}\lesssim \gamma_{n,1}\simeq L_n,\end{equation}
where $c>0$ is an absolute constant.
\end{proposition}

Proposition \ref{prop:worse-k} shows that a positive answer to Question \ref{question-intro} (actually, in the case $k=1$) is equivalent to the uniform boundedness
of the isotropic constants of all convex bodies in all dimensions (this is exactly the hyperplane conjecture). It also shows that
the question becomes ``easier" when the codimension $k$ increases, in the sense that $\gamma_{n,k}\lesssim \gamma_{n,1}$. In fact, we
can show that if $k$ is proportional to $n$ then $\gamma_{n,k}$ is bounded (this is precisely the content of Theorem \ref{th:intro-7}):

\medskip

\begin{theorem}\label{th:gamma-DP}For any $1\ls k\ls n-1$ we have
\begin{equation}\label{eq:worse-k-4}\gamma_{n,k}\ls c\sqrt{n/k}\,[\log (en/k)]^{\frac{3}{2}},\end{equation}
where $c>0$ is an absolute constant.
\end{theorem}

\begin{proof}We repeat the proof of Theorem \ref{th:intro-6} using the estimate (see \cite[Theorem~1.3]{Dafnis-Paouris-2012})
\begin{equation}\label{eq:main-7}\tilde{R}_k(K)\gr \left (\frac{c_5}{\sqrt{n/k}\,[\log (en/k)]^{\frac{3}{2}}}\right )^{kn}\end{equation}
instead of \eqref{eq:main-6}. \end{proof}

\begin{remark}\label{rem:dual-affine}\rm Let $\alpha_{n,k}$ be the largest constant $\alpha >0$ with the property that $\tilde{R}_k(K)\gr \alpha^{kn}$
for every centered convex body $K$ in ${\mathbb R}^n$. Repeating the proof of Theorem \ref{th:intro-6} or Theorem \ref{th:gamma-DP} we see that
\begin{equation}\label{eq:worse-k-5}\gamma_{n,k}\ls \frac{c_1}{\alpha_{n,k}},\end{equation}
where $c_1>0$ is an absolute constant. In particular,
\begin{equation}\gamma_{n,k}\lesssim \gamma_{n,1}\simeq L_n\lesssim \alpha_{n,1}^{-1}.\end{equation}
On the other hand, \eqref{eq:main-6} shows that $\alpha_{n,k}\gr c/L_n$
for all $1\ls k\ls n-1$, and hence $\alpha_{n,1}^{-1}\lesssim L_n$. Therefore,
\begin{equation}\gamma_{n,1}\simeq L_n\simeq \alpha_{n,1}^{-1}.\end{equation}
In other words, the question whether
\begin{equation}\label{eq:worse-k-6}\tilde{R}_k(K)\gr c^{kn}\end{equation}
for all $1\ls k\ls n-1$ which is studied in \cite{Dafnis-Paouris-2012} (see also \cite[Section~9.4]{Gardner-book}) is equivalent to the
hyperplane conjecture and to Question \ref{question-intro}.
\end{remark}

%%%%%%%%%%%%%%%%%%%%%%%%%%%%%%%%%%%%%%%%%%%%%%%%%%%%%%%%%%%%%%%%%%%%%%%%%%%%%%%%%%%%%%%%%%%%%%%%%%%%%%%%%%%%%%%%%%%%%%%
\section{Reverse inequalities in the classical positions}\label{positions}
%%%%%%%%%%%%%%%%%%%%%%%%%%%%%%%%%%%%%%%%%%%%%%%%%%%%%%%%%%%%%%%%%%%%%%%%%%%%%%%%%%%%%%%%%%%%%%%%%%%%%%%%%%%%%%%%%%%%%%%

Next, we pass to estimates for the mean value of the average section functional of hyperplane sections of $K$. We start by expressing
${\rm as}(K)$ in terms of dual mixed volumes. Note that by \eqref{eq:ovr-7} we have
\begin{equation}\label{eq:DMX-3}{\rm as}(K)=\omega_{n-1}\int_{S^{n-1}}\rho_K^{n-1}(\theta )\,d\sigma (\theta )=\frac{\omega_{n-1}}{\omega_n}\tilde{V}(K,\ldots ,K,B_2^n),
\end{equation}
and using \eqref{eq:ovr-8} we see that
\begin{equation}\label{eq:DMX-4}\int_{G_{n,,n-k}}{\rm as}(K\cap E)\,d\nu_{n-k}(E)=\omega_{n-k-1}\int_{S^{n-1}}\rho_K^{n-k-1}(\theta )\,d\sigma (\theta )
=\frac{\omega_{n-k-1}}{\omega_n}\tilde{V}(K\,[n-k-1],B_2^n\,[k+1]),\end{equation}
where $A\,[s]$ means $A,\ldots ,A$ repeated $s$-times.

\begin{theorem}\label{th:BGL-2}Let $K$ be a centered convex body in ${\mathbb R}^n$. Then,
\begin{equation}\label{eq:BGL-10}{\rm as}(K)^{k+1}\ls c^k\,|K|^k\,\int_{{\rm Gr}_{n-k}}{\rm as}(K\cap E)\,d\nu_{n-k} (E)\end{equation}
and
\begin{equation}\label{eq:BGL-10b}\int_{{\rm Gr}_{n-k}}{\rm as}(K\cap E)\,d\nu_{n-k} (E)\ls c^k\,{\rm as}(K)^{\frac{n-k-1}{n-1}}, \end{equation}
where $c>0$ is an absolute constant.
\end{theorem}

\begin{proof}From H\"{o}lder's inequality we see that
\begin{equation}\left (\int_{S^{n-1}}\rho_K^{n-1}(\theta )\,d\sigma (\theta )\right )^{k+1}\ls
\left (\int_{S^{n-1}}\rho_K^n(\theta )\,d\sigma (\theta )\right )^k\left (\int_{S^{n-1}}\rho_K^{n-k-1}(\theta )\,d\sigma (\theta )\right ),\end{equation}
which can be equivalently written as
\begin{equation}\label{eq:BGL-11}\tilde{V}(K,\ldots ,K,B_2^n)^{k+1}\ls |K|^k\,\tilde{V}(K\,[n-k-1],B_2^n\,[k+1]).\end{equation}
Taking into account \eqref{eq:DMX-3} and \eqref{eq:DMX-4} we rewrite \eqref{eq:BGL-11} in the form
\begin{equation}\label{eq:BGL-12}{\rm as}(K)^{k+1}\ls |K|^k\,\frac{\omega_{n-1}^{k+1}}{\omega_n^k\omega_{n-k-1}}\int_{{\rm Gr}_{n-k}}{\rm as}(K\cap E)\,d\nu_{n-k} (E).\end{equation}
A simple computation shows that $\frac{\omega_{n-1}^{k+1}}{\omega_n^k\omega_{n-k-1}}<c^k$ for an absolute constant $c>0$, and \eqref{eq:BGL-10} follows.

On the other hand, by H\"{o}lder's inequality,
\begin{align}\label{eq:BGL-13}&\frac{1}{\omega_{n-k-1}}\int_{{\rm Gr}_{n-k}}{\rm as}(K\cap E)\,d\nu_{n-k} (E) =\int_{S^{n-1}}\rho_K^{n-k-1}(\theta )\,d\sigma (\theta )\\
\nonumber &\hspace*{1.5cm}\ls \left (\int_{S^{n-1}}\rho_K^{n-1}(\theta )\,d\sigma (\theta )\right )^{\frac{n-k-1}{n-1}}= \left (\frac{{\rm as}(K)}{\omega_{n-1}}\right )^{\frac{n-k-1}{n-1}},
\end{align}
therefore
\begin{equation}\label{eq:BGL-14}\int_{{\rm Gr}_{n-k}}{\rm as}(K\cap E)\,d\nu_{n-k} (E)\ls \varrho_{n,k} \,{\rm as}(K)^{\frac{n-k-1}{n-1}},\end{equation}
where $\varrho_{n,k}=\omega_{n-k-1}\cdot \omega_{n-1}^{-\frac{n-k-1}{n-1}}\ls c^k$ for an absolute constant $c>0$, which gives \eqref{eq:BGL-10b}. \end{proof}

\medskip

Let $K$ be a convex body in ${\mathbb R}^n$ with $0\in {\rm int}(K)$. Recall that the radius $R(K)$ of $K$ is the smallest $R>0$ for which $K\subseteq RB_2^n$.
Using the monotonicity and homogeneity of dual mixed volumes and \eqref{eq:DMX-3} we may write
\begin{align}\label{eq:as-radius-0}{\rm as}(K) &= \frac{\omega_{n-1}}{\omega_n}\tilde{V}(K,\ldots ,K,B_2^n)\gr
\frac{\omega_{n-1}}{\omega_n R(K)}\tilde{V}(K,\ldots ,K,R(K)B_2^n)\\
\nonumber &\gr \frac{\omega_{n-1}}{\omega_n R(K)}\tilde{V}(K,\ldots ,K,K)=\frac{\omega_{n-1}}{\omega_n R(K)}\,|K|.
\end{align}
In this way we obtain the following general lower bound for ${\rm as}(K)$.

\begin{lemma}\label{lem:as-radius-2}Let $K$ be a centered convex body in ${\mathbb R}^n$. If we define
$p(K)=R(K)/|K|^{\frac{1}{n}}$ then
\begin{equation}\label{eq:as-radius-3}\frac{c\sqrt{n}}{p(K)}\ls \frac{{\rm as}(K)}{|K|^{\frac{n-1}{n}}},\end{equation}
where $c>0$ is an absolute constant.
\end{lemma}

\begin{proof}From \eqref{eq:as-radius-0} we see that
\begin{equation}\label{eq:as-radius-5}\frac{{\rm as}(K)}{|K|^{\frac{n-1}{n}}}\gr \frac{\omega_{n-1}}{\omega_n R(K)}|K|^{1/n}\gr \frac{c\sqrt{n}}{R(K)}|K|^{1/n},\end{equation}
and the lemma follows by the definition of $p(K)$. \end{proof}

\medskip

Going back to Theorem \ref{th:BGL-2} we immediately get the following.

\begin{theorem}\label{th:as-radius-3}Let $K$ be a centered convex body in ${\mathbb R}^n$. Then, for every $1\ls k\ls n-1$ we have that
\begin{equation}\label{eq:as-radius-6}\left (\frac{c_1\sqrt{n}}{p(K)}\right )^k\,{\rm as}(K)\ls |K|^{\frac{k}{n}}\,\int_{{\rm Gr}_{n-k}}{\rm as}(K\cap E)\,d\nu_{n-k}(E)\ls \left (\frac{c_2p(K)}{\sqrt{n}}\right )^{\frac{k}{n-1}}\,{\rm as}(K),\end{equation}
where $c_1, c_2>0$ are absolute constants.
\end{theorem}

\begin{proof}The left hand side inequality follows from Lemma \ref{lem:as-radius-2} and \eqref{eq:BGL-10}. We have
$$|K|^{\frac{k(n-1)}{n}}\left (\frac{c_1\sqrt{n}}{p(K)}\right )^k{\rm as}(K)\ls {\rm as}(K)^{k+1}\ls c^k|K|^k\int_{{\rm Gr}_{n-k}}{\rm as}(K\cap E)\,d\nu_{n-k}(E),$$
which implies that
$$\left (\frac{c_1\sqrt{n}}{cp(K)}\right )^k\,{\rm as}(K)\ls |K|^{\frac{k}{n}}\,\int_{{\rm Gr}_{n-k}}{\rm as}(K\cap E)\,d\nu_{n-k}(E).$$
Next, we observe that
\begin{equation}\label{eq:as-radius-7}|K|^{\frac{k}{n}}=\left (|K|^{\frac{n-1}{n}}\right )^{\frac{k}{n-1}}\ls \left (\frac{p(K){\rm as}(K)}{c\sqrt{n}}\right )^{\frac{k}{n-1}},\end{equation}
which implies that
\begin{equation}\label{eq:as-radius-8}|K|^{\frac{k}{n}}\,{\rm as}(K)^{\frac{n-k-1}{n-1}}\ls \left (\frac{c_2p(K)}{\sqrt{n}}\right )^{\frac{k}{n-1}}{\rm as}(K).\end{equation}
Then, the right hand side inequality of \eqref{eq:as-radius-6} follows from \eqref{eq:BGL-10b}
in Theorem \ref{th:BGL-2}. \end{proof}

\begin{remark}\label{rem:classical}\rm We will discuss the estimates that one can get from Theorem \ref{th:as-radius-3} if the centered convex body $K$ in ${\mathbb R}^n$
is assumed to be in some of the classical positions; we introduce these below. For a detailed presentation and references see \cite{AGA-book}.
\begin{enumerate}
\item[(i)] We say that $K$ is in minimal mean width position if $w(K)\ls w(T(K))$ for every $T\in SL(n)$. It was proved by V.~Milman and the second named author
that $K$ has minimal mean width if and only if
\begin{equation}\label{eq:notation-17}w(K)=n\int_{S^{n-1}}h_K(\theta )\langle \xi ,\theta\rangle^2d\sigma (\theta )\end{equation}
for every $\xi\in S^{n-1}$. From results of Figiel-Tomczak, Lewis and Pisier (see \cite[Chapter 6]{AGA-book})
we know that if a convex body $K$ in ${\mathbb R}^n$ has minimal mean width then $w(K)\ls c|K|^{\frac{1}{n}}\sqrt{n}\log n$.
From the general fact that $R(K)\ls c\sqrt{n}w(K)$ for every centered convex body, we conclude that $R(K)\ls c|K|^{\frac{1}{n}}n\log n$
in the minimal mean width position.
\item[(ii)] We say that $K$ is in John's position if the ellipsoid of maximal volume inscribed in $K$ is a multiple of the Euclidean unit ball $B_2^n$
and that $K$ is in L\"{o}wner's position if the ellipsoid of minimal volume containing $K$ is a multiple of the Euclidean unit ball
$B_2^n$. One can check that this holds true if and only if $K^{\circ }$ is in John's position.
The volume ratio of a centered convex body $K$ in ${\mathbb R}^n$ is the quantity
\begin{equation}\label{eq:notation-18}{\rm vr}(K)=\inf\left\{ \left(\frac{|K|}{|{\cal E}|}\right )^{\frac{1}{n}}:{\cal E}\;\hbox{is an ellipsoid and}\;
{\cal E}\subseteq K\right\}.\end{equation}
The outer volume ratio of a centered convex body $K$ in ${\mathbb R}^n$ is the quantity ${\rm ovr}(K)={\rm vr}(K^{\circ })$.
K.~Ball proved in \cite{Ball-reverse} that ${\rm vr}(K)\ls {\rm vr}(C_n)\simeq \sqrt{n}$ in the symmetric case and
${\rm vr}(K)\ls {\rm vr}(\Delta_n)\simeq \sqrt{n}$ in the not necessarily symmetric case, where $C_n=[-1,1]^n$ and $\Delta_n$ is a regular simplex in
${\mathbb R}^n$. Assume that $K$ is in John's position. Then, from a theorem of Barthe \cite{Barthe-width} we know that
if $\Delta_n$ is the regular simplex whose maximal volume ellipsoid is $B_2^n$
and $rB_2^n$ is the maximal volume ellipsoid of $K$ we have $w(r^{-1}K)\ls w(\Delta_n)\ls c\sqrt{\log n}$. Since
$|K|^{1/n}\gr r|B_2^n|^{1/n}\gr cr/\sqrt{n}$, we get
\begin{equation}R(K)\ls c\sqrt{n}w(K)=cr\sqrt{n}w(r^{-1}K)\ls cr\sqrt{n\log n}\ls c^{\prime }|K|^{\frac{1}{n}}n\sqrt{\log n}.\end{equation}
Next, assume that $K$ is in L\"{o}wner's position; we know that $R(K)B_2^n$ is the minimal volume ellipsoid of $K$, and hence
\begin{equation}R(K)|B_2^n|^{1/n}=|K|^{\frac{1}{n}}{\rm ovr}(K)=|K|^{\frac{1}{n}}{\rm vr}(K^{\circ })\ls c\sqrt{n}|K|^{\frac{1}{n}},\end{equation}
which implies that $R(K)\ls cn|K|^{\frac{1}{n}}$.
\item[(iii)] We say that $K$ has minimal surface area if $S(K)\ls S(T(K))$ for every $T\in SL(n)$. Recall that the area measure $\sigma_K$ of $K$ is the Borel measure
on $S^{n-1}$ defined by
\begin{equation}\sigma_K(A)=\lambda (\{x\in {\rm bd}(K):\;{\rm the}\;{\rm outer}
\;{\rm normal}\;{\rm to}\;K\;{\rm at}\;x\;{\rm belongs}\;
{\rm to}\;A\}),\end{equation}
where $\lambda $ is the usual surface measure on $K$. Petty proved in \cite{Petty-1961} that $K$ has minimal surface
area if and only if $\sigma_K$ satisfies the isotropic condition
\begin{equation}\label{eq:notation-19}S(K)=n\int_{S^{n-1}}\langle \xi ,\theta\rangle^2d\sigma_K (\theta )\end{equation}
for every $\xi\in S^{n-1}$. It is known that if $K$ has minimal surface area then $w(K)\ls cn|K|^{\frac{1}{n}}$ (this was observed by Markessinis, Paouris
and Saroglou in \cite{Markessinis-Paouris-Saroglou-2012}). Therefore, $R(K)\ls cn^{\frac{3}{2}}|K|^{\frac{1}{n}}$.
\item[(iv)] Finally, if $K$ is in the isotropic position then we know that $R(K)\ls |K|^{\frac{1}{n}}(n+1)L_K$.
This estimate is due to Kannan, Lov\'{a}sz and Simonovits (the asymptotically sharp bound $R(K)\ls cnL_K\,|K|^{\frac{1}{n}}$ can
be obtained with an elementary argument).
\end{enumerate}
\end{remark}

Since $R(K)$ is polynomial in $n$ for all the classical positions of a convex body $K$, from the right hand side inequality \eqref{eq:as-radius-6}
of Theorem \ref{th:as-radius-3} we obtain the next result.

\begin{theorem}\label{th:as-radius-4}Let $K$ be a centered convex body in ${\mathbb R}^n$. If $K$ is in any of the classical positions 
that we discussed in Remark $\ref{rem:classical}$, then
\begin{equation}\label{eq:as-radius-9}|K|^{\frac{k}{n}}\,\int_{{\rm Gr}_{n-k}}{\rm as}(K\cap E)\,d\nu_{n-k}(E)\ls C^k\,{\rm as}(K)\end{equation}
for all $1\ls k\ls n-1$, where $C>0$ is an absolute constant.
\end{theorem}

\begin{remark}\rm Similarly, from the left hand side inequality \eqref{eq:as-radius-6}
of Theorem \ref{th:as-radius-3} we see that if $K$ is in some of the classical positions that we discussed in Remark \ref{rem:classical}
then
\begin{equation}\label{eq:as-radius-10}c_n^{-k}\,{\rm as}(K)\ls |K|^{\frac{k}{n}}\,\int_{{\rm Gr}_{n-k}}{\rm as}(K\cap E)\,d\nu_{n-k}(E)
\end{equation}
for every $1\ls k\ls n-1$, where $c_n\simeq\sqrt{n}$ if $K$ is in L\"{o}wner's position, $c_n\simeq\sqrt{n\log n}$ if $K$ is in John's position, $c_n\simeq\sqrt{n}(\log n)$ if $K$ is in the
minimal mean width position, $c_n\simeq n$ if $K$ is in the minimal surface area position, and $c_n\simeq\sqrt{n}L_K$ if $K$ is in the isotropic position.
\end{remark}

\bigbreak

\bigskip

\bigskip

%%%%%%%%%%%%%%%%%%%%%%%%%%%%%%%%%%%%%%%%%%%%%%%%%%%%%%%%%%%%%%%%%%%%%%%%%%%%%%%%%%%%%%%%%%%%%%%%%%%%%%%%%%%%%%%%%%%%%%%%%%%%%%%%%%%%%%
\noindent {\bf Acknowledgements.}  The first named author acknowledges support from Onassis Foundation. The fourth named author
was partially supported by the US National Science Foundation grant DMS-1265155.
%%%%%%%%%%%%%%%%%%%%%%%%%%%%%%%%%%%%%%%%%%%%%%%%%%%%%%%%%%%%%%%%%%%%%%%%%%%%%%%%%%%%%%%%%%%%%%%%%%%%%%%%%%%%%%%%%%%%%%%%%%%%%%%%%%%%%%

\bigskip

\bigskip

\footnotesize
\bibliographystyle{amsplain}

\bigskip

\bigskip

\noindent \textsc{Silouanos \ Brazitikos}: Department of
Mathematics, National and Kapodistrian University of Athens, Panepistimiopolis 157-84,
Athens, Greece.

\smallskip

\noindent \textit{E-mail:} \texttt{silouanb@math.uoa.gr}

\bigskip 

\noindent \textsc{Susanna \ Dann}: Institute of Discrete Mathematics and Geometry, Vienna University of
Technology, Wiedner Hauptstrasse 8-10, 1040 Vienna, Austria.

\smallskip 

\noindent \textit{E-mail:} \texttt{susanna.dann@tuwien.ac.at}

\bigskip

\noindent \textsc{Apostolos \ Giannopoulos}: Department of
Mathematics, National and Kapodistrian University of Athens, Panepistimiopolis 157-84,
Athens, Greece.

\smallskip

\noindent \textit{E-mail:} \texttt{apgiannop@math.uoa.gr}

\bigskip

\noindent \textsc{Alexander \ Koldobsky}: Department of
Mathematics, University of Missouri, Columbia, MO 65211.

\smallskip

\noindent \textit{E-mail:} \texttt{koldobskiya@missouri.edu}

\end{document}